\documentclass[11pt]{article}

\usepackage{amsmath}
\usepackage{amsthm}
\usepackage{amssymb}
\usepackage{amscd}
\usepackage{pdfsync}
\usepackage{epsfig}
\usepackage{graphicx,color} 
 
 \newtheorem{theorem}{\sc Theorem}[section]
 \newtheorem{lemma}[theorem]{\sc Lemma}
 \newtheorem{cor}[theorem]{\sc Corollary}
 \newtheorem{prop}[theorem]{\sc Proposition}

\newtheorem{remark}[theorem]{\sc Remark}

\usepackage{latexsym}
\usepackage{url}

\def\PG{\mathrm{PG}}

\def\F{\mathbb{F}}
\def\Fq{\mathbb{F}_q}

\def\cD{\mathcal{D}}
\def\R{\mathcal{R}}

\def\GL{\mathrm{GL}}

\def\dim{\mathrm{dim}}

\def\cF{\mathcal{F}}
\def\cB{\mathcal{B}}

\def\cH{\mathcal{H}}

\DeclareMathOperator{\PGL}{PGL}

\DeclareMathOperator{\AG}{AG}

\def\cC{\mathcal{C}}
\def\cI{\mathcal{I}}
\def\cK{\mathcal{K}}

\def\cP{\mathcal{P}}
\def\cR{\mathcal{R}}
\def\cS{\mathcal{S}}
\def\cV{\mathcal{V}}

\usepackage{tikz}
\usetikzlibrary{matrix,decorations.pathreplacing}

\newcommand{\segre}[3]{\cS_{#1,#2,#3}}

\begin{document}

\title{Subspaces intersecting each element of a regulus in one point, Andr\'e-Bruck-Bose representation and clubs}
\author{Michel Lavrauw\thanks{Dipartimento di Tecnica e Gestione dei Sistemi Industriali,
Universit\`a di Padova, Stradella S. Nicola 3, 36100 Vicenza, Italy, e-mail: michel.lavrauw@unipd.it. 
The research of this author is supported by the Research Foundation Flanders-Belgium (FWO-Vlaanderen) and by a 
Progetto di Ateneo from Universit\`a di Padova (CPDA113797/11).} 
\ and Corrado Zanella\thanks{Dipartimento di Tecnica e Gestione dei Sistemi Industriali,
Universit\`a di Padova, Stradella S. Nicola 3, 36100 Vicenza, Italy, e-mail: corrado.zanella@unipd.it.
The research of this author is supported by the Italian Ministry of Education, University and Research (PRIN 2012 project ``Strutture geometriche, combinatoria e loro applicazioni'').}
}
\date{\today}
\maketitle
\begin{abstract}
In this paper results are proved with applications to the orbits of $(n-1)$-dimensional subspaces disjoint from a regulus $\cR$ of $(n-1)$-subspaces in
$\PG(2n-1,q)$, with respect to the subgroup of $\PGL(2n,q)$ fixing $\cR$.
Such results have consequences on several aspects of finite geometry.
First of all, a necessary condition for an $(n-1)$-subspace $U$ and a regulus $\cR$ of $(n-1)$-subspaces to be extendable to a Desarguesian spread is given.
The description also allows to improve results in \cite{BaJa12} 
on the Andr\'e-Bruck-Bose representation of a $q$-subline in 
$\PG(2,q^n)$. Furthermore, the results in this paper are applied to 
the classification of linear sets, in particular clubs.

A.M.S. CLASSIFICATION: 51E20

\end{abstract}

\section{Introduction}\label{sec:intro}

The $(n-1)$-dimensional projective projective space over the field $F$ is denoted by $\PG(n-1,F)$ or $\PG(n-1,q)$ if $F$ is the finite field of order $q$ (denoted by $\F_q$). If $L$ is an extension field $\F_q$, then the projective space defined by the $\F_q$-vector space induced by $L^d$ is denoted by $\PG_q(L^d)$. For further notation and general definitions employed in this paper the reader is referred to \cite{LaVaPrep,LaZa14b,Po10}.
For more information on Desarguesian spreads see \cite{BaLu11}.

This paper is structured as follows. In Section \ref{sec:2} subspaces which intersect each element of a regulus in one point are studied and a result from \cite{Burau} is generalised. Section \ref{sec:3} contains one of the main results of this paper, determining the order of the normal rational curves obtained from $n$-dimensional subspaces on an external $(n-1)$-dimensional subspace with respect to a regulus in $\PG(2n-1,q)$, obtained from a point and a subline after applying the field reduction map to $\PG(1,q^n)$. This leads to a necessary condition on the existence of a Desarguesian spread containing a subspace and regulus (Corollary \ref{cor1}). The Andr\'e-Bruck-Bose representation of sublines and subplanes of a finite projective
plane is studied in Section \ref{sec:4} and improvements are obtained
with respect to the known results  \cite{BFG,QuCa,Vi,BaJa12}. The results from the first sections are then applied to the classification problem for clubs of rank three in $\PG(1,q^n)$ in Section \ref{sec:clubs}. 
A study of the incidence structure of the clubs in $\PG(1,q^n)$ after field reduction yields to a partial
classification, concluding that the orbits of clubs under $\PGL(2,q^n)$ are at least
$k-1$, where $k$ stands for the number of divisors of $n$.
The paper concludes with an appendix discussing a result motivated by Burau \cite{Burau} for the complex numbers: the result is extended to general algebraically closed fields; a new proof is provided; and counterexamples are given to some of the arguments used in the original proof.

\section{Subspaces intersecting each element of a regulus in one point}\label{sec:2}

Let $\cR$ be a regulus of subspaces in a projective space and let $S$ be any subspace of $\langle \cR\rangle$.
Questions about the properties of the set of intersection points, which for reasons of simplicity of notation we will denote by $S\cap \cR$, often
turn up while investigating objects in finite geometry.
If $S$ intersects each element of the regulus $\cR$ in a point, then the intersection 
$S\cap \cR$ is a normal rational curve, see Lemma \ref{lemma:0}. This was already pointed out in 
\cite[p.173]{Burau} with a proof originally intended for complex projective spaces, but actually 
holding in a more general setting. The
notation of \cite{Burau} will be partly adopted.

The Segre variety representing the Cartesian product $\PG(n,F)\times\PG(m,F)$ in
$\PG((n+1)(m+1)-1,F)$ is denoted by $\segre nmF$. It is well known that
$\segre nmF$ contains two families $\segre nmF^I$ and $\segre nmF^{II}$ 
of maximal subspaces of dimensions $n$ and $m$, respectively. When convenient,
the notation $S^I$ or $S^{II}$ will be used for a subspace belonging to the first or second family.
The points of $\segre nmF$ may be represented as one-dimensional subspaces spanned by rank one 
$(m+1)\times(n+1)$ matrices.
This is the standard example of a regular embedding of product spaces, see \cite{Za96}.
Note that in the finite case it is possible to embed product spaces in projective spaces of smaller dimension (see e.g. \cite{LaShZa2013}).
A regulus $\cR$ of $(n-1)$-dimensional subspaces can also be defined as $\segre {n-1}1F^I$.

\begin{lemma}\label{lemma:0}
  Let $n>1$ be an integer, and $F$ a field.
  Let $S_t$ be a $t$-subspace of $\PG(2n-1,F)$ intersecting each $S^I\in\segre{n-1}1F^I$ in precisely one point.
  Define $\Phi=S_t\cap\segre{n-1}1F$, and
  assume $\langle\Phi\rangle=S_t$.
 Then $|F|\geq t$ and the following properties hold.
  \begin{itemize}
  \item[(i)] The set $\Phi$ is a normal rational curve of order $t$.
  \item[(ii)] Let $\Xi^I\in\segre{n-1}1F^I$.
  Then the set $S(\Phi,\Xi^I)$ of the intersections of $\Xi^I$ with all transversal lines $l^{II}$
  such that $l^{II}\cap\Phi\neq\emptyset$ is a normal rational curve of order
  $t$ or $t-1$ if $|F|=t$, and of order $t-1$ if $|F|>t$.
  \item[(iii)] If $\Phi$ is contained in a subvariety $\cS_{t-1,1,F}$ of $\cS_{n-1,1,F}$,
  then homogeneous coordinates can be chosen
  such that $\Phi$ is represented parametrically by
  \begin{equation}\label{proj-1}
    \left\langle\begin{pmatrix}
    y_0^t& y_0^{t-1}y_1 &\ldots &y_0y_1^{t-1}\\
    y_0^{t-1}y_1&y_0^{t-2}y_1^2&\ldots &y_1^t\end{pmatrix}\right\rangle,\quad 
    (y_0,y_1)\in(F^2)^*,
  \end{equation}
  and $S(\Phi,\Xi^I)$, for $z_0$, $z_1$ depending only on $\Xi^I$, by
  \begin{equation}\label{proj-2}
    \left\langle\begin{pmatrix}
    y_0^{t-1}z_0& y_0^{t-2}y_1z_0 &\ldots &y_1^{t-1}z_0\\
    y_0^{t-1}z_1&y_0^{t-2}y_1z_1&\ldots &y_1^{t-1}z_1\end{pmatrix}\right\rangle,\quad 
    (y_0,y_1)\in(F^2)^*.
  \end{equation}
  \end{itemize}
\end{lemma}
\begin{proof}
  \textit{(i), (iii)} The proof in \cite[Sect.41 no.3]{Burau}, which is offered for $F=\mathbb C$, works exactly the
  same provided that $|F|>t$ or, more generally, that $\Phi$ is contained in some subvariety
  $\cS_{t-1,1,F}$ of $\cS_{n-1,1,F}$.
  In case $|F|\le t$, the size of $\Phi$ being $|F|+1$ implies $|F|=t$, so $\Phi$ is just a set
  of $t+1$ independent points in a subspace isomorphic to $\PG(t,t)$, hence $\Phi$ is a normal rational
  curve of order $t$.
  
  \textit{(ii)} The case $|F|>t$ is proved in \cite{Burau} immediately after the corollary at p.\ 175.
  If $|F|\leq t$, then $|F|=t$ and two cases are possible.
  If $\Phi$ is contained in some $\cS_{t-1,1,F}\subseteq\cS_{n-1,1,F}$, Burau's proof is still valid as
 was mentioned in case (ii); so, $S(\Phi,\Xi^I)$ is a normal rational curve of order $t-1=|F|-1$.
  Otherwise $S(\Phi,\Xi^I)$ is an independent $(t+1)$-set, hence a normal rational curve of order $|F|$.
\end{proof}
\begin{remark}
\emph{
  If $|F|=t$ both cases in Lemma \ref{lemma:0} \textit{(ii)} can occur. The following two examples use the 
  Segre embedding $\sigma=\sigma_{t-1,1,F}$ of the product space $\PG(t-1,t) \times \PG(1,t)$ in $\PG(2t-1,t)$.
  Let $\{s_0,s_1,\ldots,s_t\}$ be the set of points on $\PG(1,t)$ and suppose $\{r_0,r_1,\ldots,r_t\}$ is a set of 
  $t+1$ points in $\PG(t-1,t)$. Put $\Xi^I=\sigma(\PG(1,t)\times s_0)$ and $\Phi:=\{\sigma(r_i\times s_i)~:~ i=0,1,\ldots, t\}$.
  Then $\Phi$ consists of $t+1$ points on the Segre variety $\cS_{t-1,1,F}$. Depending on the set $\{r_0,r_1,\ldots,r_t\}$ one obtains the two cases described in Lemma \ref{lemma:0} \textit{(ii)}.
  \begin{itemize}
  \item[a.] If $\{r_0,r_1,\ldots,r_t\}$ is a frame of a hyperplane of $\PG(t-1,t)$ then $\Phi$ generates a 
  $t$-dimensional subspace of $\PG(2t-1,t)$ intersecting $\cS_{t-1,1,F}$ in $\Phi$ and $S(\Phi,\Xi^I)$ is a normal rational curve of
  order $t-1$.
  \item[b.] If $\{r_0,r_1,\ldots,r_t\}$ generates $\PG(t-1,t)$ then $\Phi$ generates a $t$-dimensional
  subspace of $\PG(2t-1,t)$ intersecting $\cS_{t-1,1,F}$ in $\Phi$ and $S(\Phi,\Xi^I)$ is a normal rational curve of
  order $t$.
  \end{itemize}
}
\end{remark}

\begin{remark}\label{proj-is-proj}\emph{
  By (\ref{proj-1}) and (\ref{proj-2}), the map $\alpha:\,\Phi\rightarrow S(\Phi,\Xi^I)$ defined by the condition
  that $X$ and $X^\alpha$ are on a common line in 
  $\cS_{n-1,1,F}^{II}$ is related to a projectivity between the parametrizing
  projective lines.
  Such an $\alpha$ is also called a {\it projectivity}.
  }
\end{remark}

\section{The order of normal rational curves contained in $\cS_{n-1,1,q}$}\label{sec:3}
Here $n\ge2$ is an integer.
The field reduction map $\cF_{m,n,q}$ from $\PG(m-1,q^n)$ to $\PG(mn-1,q)$ will also be denoted by $\cF$.
If $S$ is a set of points, in $\PG(m-1,q^n)$, then $\cF(S)$ is a set of subspaces, whose union, as a set of points
will be denoted by $\tilde\cF(S)$.
The $\F_{q^h}$-span of a subset $b$ of $\PG(d,q^n)$ is denoted by $\langle b\rangle_{q^h}$.
\begin{prop}\label{prop:grado_pto}
  Let $b$ be a $q$-subline of $\PG(1,q^n)$, and let $\Theta\not\in b$ be a point of $\PG(1,q^n)$.
  Let $1$, $\zeta$ and $1$, $\zeta'$ be homogeneous coordinates of $\Theta$ with respect to two
  reference frames for $\langle b\rangle_{q^n}$, each of which consists of three points of $b$.
  Then $\Fq(\zeta)=\Fq(\zeta')$.
\end{prop}
\begin{proof}
  Homogeneous coordinates of a point in both reference frames, say $(x_0,x_1)$ and $(x_0',x_1')$,
  are related by an equation of the form $\rho(x_0'\ x_1')^T=A(x_0\ x_1)^T$, $\rho\in\F^*_{q^n}$,
  $A\in\GL(2,q)$.
  Hence $(\rho\ \rho\zeta')^T=A(1\ \zeta)^T$ and this implies $\zeta'\in\Fq(\zeta)$.
  The proof of $\zeta\in\Fq(\zeta')$ is similar.
\end{proof}
By Proposition\ \ref{prop:grado_pto}, the \textit{degree of a point over a $q$-subline} $b$
in a finite projective space $\PG(d,q^n)$, $[\Theta:b]=[\Fq(\zeta):\Fq]$
for $\Theta\in\langle b\rangle_{q^n}\setminus b$, $[\Theta:b]=1$ for $\Theta\in b$, is
well-defined.
This $[\Theta:b]$ also equals the minimum integer $m$ such that a subgeometry $\Sigma\cong\PG(d,q^m)$
exists containing both $b$ and $\Theta$.
\begin{prop}\label{zero-}
  Any $n$-subspace of $\PG(2n-1,q)$ containing an $(n-1)$-subspace 
  $S^I\in\cS_{n-1,1,q}^I$ intersects $\cS_{n-1,1,q}$
  in the union of $S^I$ and a line in $\cS_{n-1,1,q}^{II}$.
\end{prop}

\begin{theorem}\label{zero}
  Let $b$ be a $q$-subline of $\PG(1,q^n)$, and $\Theta\not\in b$ a point of $\PG(1,q^n)$.
  Then in $\PG(2n-1,q)$ any $n$-subspace $\cH$ containing $\cF(\Theta)$ intersects the Segre variety
  $\cS_{n-1,1,q}=\tilde{\cF}(b)$, in a normal rational curve whose order is $\min\{q,[\Theta:b]\}$.
\end{theorem}
\begin{proof}
  Set $L=\F_{q^n}$, $F=\F_q$.
  Without loss of generality, $\PG(2n-1,q)=\PG_q(L^2)$,
  $\cF(b)=\{L(x,y)\mid(x,y)\in(F^2)^*\}$\footnote{For $x,y\in L$, $F(x,y)=\langle(x,y)\rangle_q$, and
  $L(x,y)=\langle(x,y)\rangle_{q^n}$.}, and $\Theta=L(1,\xi)$ with $[F(\xi):F]=[\Theta:b]$.
  The $n$-subspace $\cH$ intersects $L(1,0)$ in one point $Y$ of the form
  $Y=F(\theta,0)$, $\theta\in L^*$.
  For any $x\in F$, seeking for the intersection
  $\langle\cF(\Theta),Y\rangle_q\cap L(x,1)$, or
  \[
    \langle L(1,\xi),F(\theta,0)\rangle_q\cap L(x,1)
  \]
  gives two equations in $\alpha,\beta\in L$:
  \[
    \alpha+\theta=\beta x,\quad \alpha\xi=\beta,
  \]
  whence $\beta=\theta(x-\xi^{-1})^{-1}$.
  The intersection point is then $F\left(x\theta(x-\xi^{-1})^{-1},\theta(x-\xi^{-1})^{-1}\right)$.
  So, for $\Xi=L(0,1)$, the set  of the intersections of  
  $\Xi$ with all lines in $\cS_{n-1,1,q}^{II}$ which meet
  $\cH$ is
  \[
    S(\cH\cap\cS_{n-1,1,q},\Xi)=
    \{F(0,\theta(x-\xi^{-1})^{-1})\mid x\in \cF_q\}\cup\{F(0,\theta)\}.
  \]
  This $S(\cH\cap\cS_{n-1,1,q},\Xi)$ is obtained by inversion from the line joining the points 
  $F(0,\theta^{-1})$
  and $F(0,\theta^{-1}\xi^{-1})$.
  By \cite[Theorem 5]{LaZa14}, $\cC_Y$ is a normal rational curve of order 
  $\delta'=\min\{q,[F(\xi^{-1}):F]-1\}=\min\{q,[\Theta:b]-1\}$.
  Now apply lemma \ref{lemma:0}
  for $S_t=\langle\cH\cap\cS_{n-1,1,q}\rangle_q$:
  if $t\ge q$, then $t=q$ and $\delta'=q$ or $\delta'=q-1$, so $[\Theta:b]\ge q$
  and $t=\min\{q,[\Theta:b]\}$.
  If on the contrary $t<q$, then $t-1=\delta'=[\Theta:b]-1$, so $t=[\Theta:b]$
  and $t=\min\{q,[\Theta:b]\}$ again.
\end{proof}

An important consequence of the above result answers the question of the existence of a 
Desarguesian spread containing a given regulus $\cR$ and a subspace disjoint from $\cR$.

\begin{cor}\label{cor1}
If a regulus $\cR=\cS_{n-1,1,q}$ and an $(n-1)$-dimensional subspace $U$, disjoint from $\cR$, in $\PG(2n-1,q)$ are contained in a Desarguesian spread then there is an integer $c$ such that any $n$-subspace $\cH$ containing $U$ intersects $\cR$ in a normal rational curve of order $c$.
\end{cor}

The following remark illustrates that this necessary condition is not always satisfied.

\begin{remark}
  For $n>2$ by using the package FinInG \cite{FinInG} of GAP \cite{GAP} examples can be given
  of $(n-1)$-subspaces disjoint from $\cS_{n-1,1,q}$ contained in $n$-subspaces intersecting
  the Segre variety in normal rational curves of distinct orders. We include one explicit example.
  Let $q=4$, $\F_q=\F_2(\omega)$, with $\omega^2+\omega+1=0$. Let
$\R$ be the regulus of $3$-dimensional subspaces of $\PG(7,4)$ obtained from the 
standard subline $\PG(1,q)$ in $\PG(1,q^4)$, and put
$$S_3:=\langle ( 1, 0, 0, 0, \omega^2, 1, 0, 1 ),
  ( 0, 1, 0, 0, 1, \omega^2, 0, \omega^2 ), 
  ( 0, 0, 1, 0, 0, \omega, 1, \omega ), 
  ( 0, 0, 0, 1, \omega^2, \omega^2, \omega, 1 )
  \rangle.
$$
Then $S_3$ is a three-dimensional subspace disjoint from the regulus $\R$. Moreover,
the 4-dimensional subspace $\langle S_3,(1,0,0,0,0,0,0,0)\rangle$ intersects the regulus $\R$
in a normal rational curve of degree 4, while the 4-dimensional
subspace $\langle S_3,(0,1,0,\omega^2,0,0,0,0)\rangle$ intersects $\R$ in a conic.
\end{remark}

\section{Andr\'e-Bruck-Bose representation}\label{sec:4}

The Andr\'e-Bruck-Bose representation of a Desarguesian affine plane of order $q^n$ is related to the image of $\PG(2,q^n)$, under the field reduction map $\cF$, by means of the
following straightforward result.
\begin{prop}\label{bb}
  Let $\cD$ be the Desarguesian spread in $\PG(3n-1,q)$ obtained after applying the field reduction map $\cF$ to the set of points of $\PG(2,q^n)$, $ l_\infty$ a line in
  $\PG(2,q^n)$, and $\cK$ a $(2n)$-subspace of $\PG(3n-1,q)$, containing the spread $\cF( l_\infty)$.
  Take $\PG(2,q^n)\setminus l_\infty$ and $\cK\setminus\langle\cF( l_\infty)\rangle_q$ as representatives of
  $\AG(2,q^n)$ and $\AG(2n,q)$, respectively.
  Then the map $\varphi:\,\AG(2,q^n)\rightarrow\AG(2n,q)$ defined by $\varphi(X)=\cF(X)\cap\cK$ for any
  $X\in\AG(2,q^n)$ is a bijection, mapping lines of $\AG(2,q^n)$ into $n$-subspaces of $\AG(2n,q)$
  whose $(n-1)$-subspaces at infinity belong to the spread $\cF( l_\infty)$.
\end{prop}

The notation in Proposition \ref{bb} is assumed to hold in the whole section.
The following result improves \cite[Theorems 3.3 and 3.5]{BaJa12}, by determining the order of
the involved normal rational curves.
\begin{theorem}\label{thm:q-subline}
  Let $b$ be a $q$-subline of $\PG(2,q^n)$, not contained in $ l_\infty$.
  Set $\Theta=\langle b\rangle_{q^n}\cap l_\infty$.
  Then the Andr\'e-Bruck-Bose representation
  $\varphi(b\setminus l_\infty)$ is the affine part of a normal rational curve whose order
  is $\delta=\min\{q,[\Theta:b]\}$.
  More precisely, if $\delta=1$, then $\varphi(b\setminus l_\infty)$ is an affine line;
  if $\delta>1$, then $b\cap l_\infty=\emptyset$, and $\varphi(b)$ is a normal rational curve
  with no points at infinity.
\end{theorem}
\begin{proof}
  The intersection $\cH=\langle\cF(b)\rangle_q\cap\cK$ is an $n$-space containing $\cF(\Theta)$,
  and contained in the span of the Segre variety $\segre{n-1}1q=\tilde\cF(b)$.
  The result follows from Proposition \ref{zero-} and Theorem \ref{zero}.
\end{proof}
The results in \cite[Theorems 3.3 and 3.5]{BaJa12} also characterize the normal rational
curves arising from $q$-sublines in $\AG(2,q^n)$.

In \cite{BFG,QuCa,Vi} for $n=2$ 
and \cite[Theorem 3.6 (a)(b)]{BaJa12} for any $n$ the Andr\'e-Bruck-Bose  representation
of a $q$-subplane tangent to a line at the infinity is described.
Further properties are stated in the following theorem:
\begin{theorem}
  Let $B$ be a $q$-subplane of $\PG(2,q^n)$ that is tangent to $l_\infty$ at the point $T$.
  Let $b$ be a line of $B$ not through $T$, $\Theta=\langle b\rangle_{q^n}\cap l_\infty$, 
  and $\delta=\min\{q,[\Theta:b]\}$.
  Then there are a normal rational curve $\cC_0$ of order $\delta$ in the $n$-subspace 
  $\varphi(\langle b\rangle_{q^n})$,
  a normal rational curve $\cC_1\subset\cF(T)$ of order $\delta'$, with
  \begin{equation}\label{eq:delta1}
    \delta'\ \left\{\begin{array}{ll}=[\Theta:b]-1&\mbox{ for }q>[\Theta:b]\\
    \in\{q-1,q\}&\mbox{otherwise,}\end{array}\right.
  \end{equation}
  and a projectivity $\kappa:\,\cC_0\rightarrow\cC_1$ (in the sense of Remark \ref{proj-is-proj}), such that
  $\varphi(B\setminus l_\infty)$ is the ruled surface union of all lines $XX^\kappa$ for $X\in\cC_0$.
\end{theorem}
\begin{proof}
  By Theorem \ref{thm:q-subline}, $\cC_0:=\varphi(b)$ is a normal rational curve of order $\delta$ in the $n$-subspace 
  $\varphi(\langle b\rangle_{q^n}\setminus l_\infty)$, 
  and for any $P=\varphi(X)\in\cC_0$, the 
  subline $TX$ of $B$ corresponds to an affine line $PP^\kappa$ with
 $P^\kappa\in\cF(T)$ 
  at infinity. 
  Define $\cC_1=\{P^\kappa\mid P\in\cC_0\}$.
  
  By the field reduction map $\cF=\cF_{3,n,q}$, the subplane $B$ is mapped to $\cF(B)$ which is the set of all maximal subspaces
  of the first family in $\cS_{n-1,2,q}\subset\PG(3n-1,q)$.
  The vector homomorphism
  \[
    (\lambda,v)\in \F_{q^n}\times \F_q^3\mapsto \lambda\otimes_{\Fq}v
  \]
  corresponds to a projective embedding $g:\,\PG(n-1,q)\times B \rightarrow\cS_{n-1,2,q}$ whose image is 
  $\cS_{n-1,2,q}$, and such that $\cF(X)=(\PG(n-1,q)\times X)^g$ for any point $X$ in $B$.
  It holds $\varphi(B\setminus l_\infty)=\cS_{n-1,2,q}\cap\cK\setminus\cF(T)$.
  For any point $U$ in $B$ define
  \[
    \kappa_U:\,(X,Y)^g\in\cS_{n-1,2,q}\mapsto(X,U)^g\in\cF(U).
  \]  
  Note that for any  $Y\in B$, the restriction of $\kappa_U$ to $\cF(Y)$ is a projectivity.
  For any $U\in b $, using the notation from Lemma \ref{lemma:0} it 
  holds $\cC_0^{\kappa_U}=S(\cC_0,\cF(U))$, and as a consequence, 
  $\cC_0^{\kappa_U}$ is a normal rational curve of order $\delta'$
  as in (\ref{eq:delta1}).
 Now, since for any $P\in\cC_0$, say $P=(X_P,Y_P)^g$, the points $P$, $P^\kappa$ and $P^{\kappa_T}$ are on the plane 
  $(X_P\times B)^g\in\segre{n-1}2q^{II}$, and $P^\kappa, P^{\kappa_T}\in \cF(T)$, it follows that
  $P^\kappa=P^{\kappa_T}$. 
  It also follows that $\cC_1=\cC_0^{{\kappa_{U}}{\kappa_{T}}}=S(\cC_0,\cF(U))^{\kappa_{T}}$, and hence $\cC_1$ is a normal rational curve of order $\delta'$ as in (\ref{eq:delta1}).
 Finally, $\kappa_{U}:\,\cC_0\rightarrow S(\cC_0,\cF(U))$ is a projectivity
  as defined in Remark \ref{proj-is-proj}, and hence so is $\kappa$.
\end{proof}

\section{On the classification of clubs}\label{sec:clubs}

An \textit{$\Fq$-club} (or simply a club) in $\PG(1,q^n)$ is an $\Fq$-linear set of rank three,
having a point of weight two, called the \textit{head} of the club.
An $\Fq$-club has $q^2+1$ points, and the non-head points have weight one.
From now on it will be assumed that $n>2$.
The next proposition is a straightforward consequence of the representation of linear sets as
projections of subgeometries \cite[Theorem 2]{LuPo04}.
\begin{prop}\label{prop:LPclub}
Let $L$ be an $\Fq$-club in $\PG(1,q^n)\subset\PG(2,q^n)$.
Then there are a $q$-subplane $\Sigma$ of $\PG(2,q^n)$, a $q$-subline $b$ in $\Sigma$, and a point
$\Theta\in\langle b\rangle_{q^n}\setminus b$, such that $L$ is the projection of $\Sigma$ from
the center $\Theta$ onto the axis $\PG(1,q^n)$.
\end{prop}

As before the notation $\cF$ and $\tilde{\cF}$ is used, where $\cF=\cF_{2,n,q}$ denotes the field reduction map from $\PG(1,q^n)$ to $\PG(2n-1,q)$.
\begin{prop}\label{prop:maximal_subspaces}
  Let $L$ be an $\Fq$-club of $\PG(1,q^n)$ with head $\Upsilon$.
  Then $\tilde\cF(L)$ contains two collections of subspaces, say $F_1$ and $F_2$, satisfying the
  following properties.
  \begin{itemize}
  \item[(i)] The subspaces in $F_1$ are $(n-1)$-dimensional, are pairwise disjoint, 
  and any subspace in $F_1$ is disjoint from $\cF(\Upsilon)$.
  \item[(ii)] Any subspace in $F_2$ is a plane and intersects $\cF(\Upsilon)$ in precisely a line.
  \item[(iii)] Any point of $\cF(\Upsilon)$ belongs to exactly $q+1$ planes in $F_2$.
  \item[(iv)] If $L$ is not isomorphic to $\PG(1,q^2)$,
  and $l$ is any line of $\PG(2n-1,q)$ contained in $\tilde\cF(L)$, then $l$ is contained in $\cF(\Upsilon)$
  or in a subspace in $F_1\cup F_2$.
  \end{itemize}
\end{prop}
\begin{proof}
The assumptions imply the existence of $\Sigma$ and a $q$-subline $b$ in $\Sigma$ as in Proposition \ref{prop:LPclub}.
  The assertions are a consequence of the fact that $\tilde\cF(\Sigma)$ is a Segre variety $\cS_{n-1,2,q}$ in $\PG(3n-1,q)$.
  Let
  \[
    p_1:\,\PG(2,q^n)\setminus\Theta\rightarrow\PG(1,q^n)
  \]
  be the projection with center $\Theta$, associated with
  \[
    p_2:\PG(3n-1,q)\setminus\cF(\Theta)\rightarrow\PG(2n-1,q).
  \]
  The collections $F_1$ and $F_2$ are defined as follows:
  \[
    F_1=\{\cF(p_1(X))\mid X\in\Sigma\setminus b\}=\cF(L)\setminus\cF(\Upsilon),\quad
    F_2=\{p_2(V^{II})\mid V^{II}\in\tilde\cF(\Sigma)^{II}\}.
  \]  
  The assertion \textit{(i)} is straightforward, as well as $\dim(V)=2$ for any $V\in F_2$.
  For any $V^{II}\in\tilde\cF(\Sigma)^{II}$, the intersection
  $V^{II}\cap\langle\tilde\cF(b)\rangle_q$ is a line, and this with
  $p_2^{-1}(\cF(\Upsilon))=\langle\tilde\cF(b)\rangle_q\setminus\cF(\Theta)$ implies 
  the second assertion in \textit{(ii)}.
  Next, let $P$ be a point in $\cF(\Upsilon)$.
  A plane $V=  p_2(V^{II})$ contains $P$ if, and only if, $V^{II}$ intersects 
  the $n$-subspace $\langle\cF(\Theta),P\rangle_q$,
  that is, $V^{II}$ intersects the normal rational curve
  $\cS_{n-1,2,q}\cap\langle\cF(\Theta),P\rangle_q$; this implies \textit{(iii)}.
  
  Assume that a line $l\subset\tilde\cF(L)$ exists which is neither contained
  in $\cF(\Upsilon)$, nor in a $T\in F_1\cup F_2$. 
  Let $Q$ be a point in $l\setminus\cF(\Upsilon)$, and let $V\in F_2$ such that $Q\in V$.
  It holds $L=\cB(V)$.
  Then $\cB(l)$ is a $q$-subline of $L$.
  Suppose that a line $l'$ in $V$ exists such that $\cB(l')=\cB(l)$.
  Since $\cB(Q)\neq\cB(Q')$ for any $Q'\in V$, $Q'\neq Q$, the line $l'$ contains $Q$.
  Then $l$, $l'$ are two distinct transversal lines in $\cB(l)^{II}$, a contradiction.
  Hence $\cB(l')\neq \cB(l)$ for any line $l'$ in $V$,  
  that is, $\cB(l)$ is a so-called \textit{irregular subline} \cite{LaVV10}.
  By \cite[Corollary 13]{LaVV10}, no irregular subline exists in $L$, and this contradiction implies \textit{(iv}).
\end{proof}

\begin{prop}\label{prop:struttura}
  Let $L$ be an $\Fq$-club with head $\Upsilon$.
  Let $\Theta$ be the point and $b$ be the subline as defined in Proposition \ref{prop:LPclub}.
  Then for any point $X$ in $\cF(\Upsilon)$, the intersection lines of $\cF(\Upsilon)$ with any $q$ distinct planes
  in $F_2$ containing $X$ span an $s$-dimensional subspace, where
  \begin{itemize}
    \item[(i)] $s=[\Theta:b]-1$ if $q>[\Theta:b]$;
    \item[(ii)] $s\in\{q-1,q\}$  if $q\le[\Theta:b]$.
  \end{itemize}
\end{prop}
\begin{proof} Let $p_2$ be the projection map as defined in the proof of Proposition \ref{prop:maximal_subspaces}, $X=p_2(P)$, and $\cH=\langle\cF(\Theta),P\rangle_q$.
  For any plane $V=p_2(V^{II})$, it holds $X\in V$ if, and only if $V^{II}\cap\cH\neq\emptyset$.
  The intersection $\cH\cap\tilde\cF(b)$ is a normal rational curve of order $\min\{q,[\Theta:b]\}$ 
  (cf.\ Theorem \ref{zero}).
  Let $V_0=p_2(V_0^{II})$ be the unique plane of $F_2$ through $X$ distinct from the $q$ planes 
  chosen in the assumptions
  (cf.\ Proposition \ref{prop:maximal_subspaces}).
  Let $Q=\tilde\cF(b)\cap V_0^{II}$; $\cB(Q)$ is an $(n-1)$-subspace of   $\tilde\cF(b)^I$.
  Such $\cB(Q)$ is mapped onto $\cB(X)=\cF(\Upsilon)$ by $p_2$.
  Assume $V_i=p_2(V_i^{II})$, $i=1,2,\ldots,q$, are the $q$ planes chosen in the assumptions.
  Any $V_i^{II}$, $i=1,2,\ldots,q$, intersects $\cH$, hence $V_i^{II}\cap\cB(Q)$
  is the intersection of $\cB(Q)$ with a transversal line of $\tilde\cF(b)$
  intersecting the normal rational curve $\cH\cap\tilde\cF(b)$.
  By Lemma \ref{lemma:0} \textit{(ii)}, the set
  \[
    S=\{V_i^{II}\cap\cB(Q)\mid i=1,2,\ldots,q\}\cup\{Q\}
  \]
  is a normal rational curve of order $s$ where $s$ takes the values as stated in $(i)$ and $(ii)$.
  Since $V_i\cap\cF(\Upsilon)$ is the line through $X$ and a point of $p_2(S)$, distinct from $X$,
  the span of the intersection lines is the same as the span of $p_2(S)$.
\end{proof}

\begin{theorem}\label{thm:num_orb}
  Let $\cI_{n,q}$ be the set of integers $h$ dividing $n$ and such that $1<h<q$.
  For any $h\in\cI_{n,q}$, let $L_h$ be the linear set obtained by projecting
  a $q$-subplane $\Sigma$ of $\PG(2,q^n)$ from a point $\Theta_h$ collinear with a
  $q$-subline $b$ in $\Sigma$ and such that $[\Theta_h:b]=h$.
  Then the set $\Lambda=\{L_h\mid h\in\cI_{n,q}\}$ contains $\Fq$-clubs in $\PG(1,q^n)$ 
  all belonging to distinct orbits under $\PGL(2,q^n)$.
\end{theorem}
\begin{proof}
  If $n$ is odd, then no club is isomorphic to $\PG(1,q^2)$.
  So, by Proposition \ref{prop:maximal_subspaces} \textit{(iv)},
  the families $F_1$ and $F_2$ are uniquely determined.
  The thesis is a consequence of Proposition \ref{prop:struttura}, taking into account that if
  $L$ and $L'$ are projectively equivalent, then $\tilde\cF(L)$ and $\tilde\cF(L')$ are projectively equivalent in
  $\PG(2n-1,q)$.
  
  In order to deal with the case $n$ even, it is enough to show that in $\Lambda$ at most one club
  is isomorphic to $\PG(1,q^2)$.
  So assume $L_h\cong\PG(1,q^2)$.
  Then $\tilde\cF(L_h)$ has a partition $\cP_1$ in $(n-1)$-subspaces, and a partition $\cP_2$ in 3-subspaces.
  From \cite[Lemma 11]{LaVV10} it can be deduced that any line contained in $\tilde\cF(L_h)$
  is contained in an element of $\cP_1$ or $\cP_2$.
  The intersections of a subspace $U$ of a family $\cP_i$ with the elements of the other family form a line
  spread of $U$.
  Hence all planes in $F_2$ are contained in 3-subspaces of $\cP_2$, and all planes of
  $F_2$ through a point $X$ in $\cF(\Upsilon)$ meet $\cF(\Upsilon)$ in the same line.
  By Proposition \ref{prop:struttura} this implies $h=2$.  
\end{proof}

\textbf{Acknowledgement.}
The authors thank Hans Havlicek for his helpful remarks in the preparation of this paper.

\appendix
\section{Appendix: On a result in \cite{Burau}}

In \cite[p.175]{Burau} the following result \textit{(Korollar)} is stated for $F=\mathbb C$.

\begin{cor}\label{cor:1}
  Let $F$ be an algebraically closed field.  
  If an $s$-subspace $S_s$ of $\PG(2s-1,F)$ meets all $S^I\in\segre{s-1}1F^I$ only in points, 
  then such points span $S_s$.
\end{cor}

In \cite{Burau} the previous result is seemingly proved using methods valid
in any field with enough elements. However such a generalisation would contradict Theorem 3.3.
In the opinion of the authors the proof in \cite{Burau} is obtained using an erroneous argument. As a matter of fact, it is claimed in the proof at page 174 that 
the assumption $\langle\Phi\rangle=S_s$
  is not used.
  However the contradiction $S_s\subset\langle\cS_{s-2,1,\mathbb C}\rangle$ is inferred from 
  $\Phi\subset\cS_{s-2,1,\mathbb C}$.

A further counterexample, which exists whenever a hyperbolic quadric $Q^+(3,F)$ in a three-dimensional projective space admits an external line (a condition which is not met when the field $F$ is algebraically closed) is the following. If $\ell$ is the line 
corresponding to the two-dimensional vector space
  $\langle e_1\rangle \otimes \langle e'_1,e'_2\rangle$ and $m$ is a line
  external to the hyperbolic quadric obtained by the intersection of the Segre variety $\cS_{2,1,F}$ with the 3-space corresponding to the vector space $\langle e_2,e_3\rangle \otimes \langle e'_1,e'_2\rangle$, then the 3-dimensional subspace $\langle \ell,m\rangle$
  intersects $\cS_{2,1,F}$ in the line $\ell$ belonging to $\segre21F^{II}$.

For the sake of completeness, a proof for corollary \ref{cor:1} is given.
\begin{proof}[Proof of corollary \ref{cor:1}]
Define
\begin{equation}\label{eq}
  S_t=\langle S_s\cap \segre{s-1}1F\rangle, ~t=\mbox{dim}S_t
\end{equation}
and suppose $t<s$.
It is proved in \cite[p.173 (6)]{Burau} that $S_t\subset\langle \segre{t-1}1F\rangle$ for some
$\segre{t-1}1F\subset\segre{s-1}1F$.

Note that $S_s\cap\langle \segre{t-1}1F\rangle=S_t$; otherwise, comparing dimensions, $S_s$ would intersect each
$S^I\in\segre{t-1}1F$  in more than one point. 
Now choose
\begin{itemize}
  \item a subspace $S_{s-t-1}\subset S_s$ such that $S_{s-t-1}\cap\langle\segre{t-1}1F\rangle=\emptyset$;
  \item a Segre variety $\segre{s-t-1}1F\subset\segre{s-1}1F$, such that 
  $\langle \segre{s-t-1}1F\rangle\cap\langle\segre{t-1}1F\rangle=\emptyset$;
  \item two distinct $A^I,B^I\in\segre{s-t-1}1F^I$.
\end{itemize}
Since $\langle\segre{s-t-1}1F\rangle$ and $\langle\segre{t-1}1F\rangle$ are complementary subspaces of
$\langle \segre{s-1}1F\rangle$, a projection map
\[
  \pi:\,\langle \segre{s-1}1F\rangle\setminus\langle\segre{t-1}1F\rangle\rightarrow\langle\segre{s-t-1}1F\rangle
\]
is defined by $\pi(P)=\langle P\cup \segre{t-1}1F\rangle\cap\langle\segre{s-t-1}1F\rangle$.

Now suppose $\pi(S_{s-t-1})\cap\segre{s-t-1}1F=\emptyset$.
In $\langle \segre{s-t-1}1F\rangle$ consider
\begin{itemize}
\item the regulus $\cR$ corresponding to $\segre{s-t-1}1F^I$, and the projectivity
$\kappa:\,A^I\rightarrow B^I$ such that, for any $P\in A^I$,
the line $\langle P,\kappa(P)\rangle$ belongs to $\segre{s-t-1}1F^{II}$;
\item the regulus $\cR'$ containing $A^I$, $B^I$ and $\pi(S_{s-t-1})$, and the 
projectivity
$\kappa':A^I\rightarrow B^I$ such that, for any $P\in A^I$,
the line $\langle P,\kappa'(P)\rangle$ is a transversal line of $\cR'$.
\end{itemize}
Since $F$ is an algebraically closed field, $\kappa'^{-1}\circ\kappa$ has a fixed point $P$.
Therefore $\kappa(P)=\kappa'(P)$, so $\cR$ and $\cR'$ have a common transversal.
This contradicts $\pi(S_{s-t-1})\cap \segre{s-t-1}1F=\emptyset$.
So, a point $P\in S_{s-t-1}$ exists such that
$\pi(P)\in\segre{s-t-1}1F$.

Next, let $C^I\in\segre{s-1}1F^I$ be such that $\pi(P)\in C^I$, and $Q$ the point in 
$\langle \segre{t-1}1F\rangle$ such that $Q$, $P$, and $\pi(P)$ are collinear.
If $Q\in S_t$, then $\pi(P)\in S_s$, a contradiction; also $Q\in C^I$ leads to a contradiction 
(since it implies $P\in C^I$).
So $Q\not\in S_t\cup C^I$ and by a dimension argument two points $Q_1\in C^I\setminus S_t$
and $Q_2\in S_t\setminus C^I$ exist such that $Q$, $Q_1$ and $Q_2$ are collinear:
they are on the unique line through $Q$ meeting both $C^I\cap\langle \segre{t-1}1F\rangle$
and a $(t-1)$subspace of $S_t$ disjoint from $C^I$.

The plane $\langle P,Q_1,Q_2\rangle$ contains the lines $PQ_2\subset S_s$ and
$\pi(P)Q_1\subset \segre{s-1}1F$ which meet outside $\langle \segre{t-1}1F\rangle$.
This is again a contradiction.
\end{proof}

\end{document}